\documentclass{amsart}    
\usepackage{amsmath} 
\usepackage{paralist}
\usepackage{cancel} 
\usepackage{graphics} 
\usepackage{epsfig} 
\usepackage{graphicx}  
\usepackage{epstopdf} 
\usepackage{mathrsfs,dsfont}
\usepackage[colorlinks=true]{hyperref}
\hypersetup{urlcolor=blue, citecolor=red}

\newcommand{\rd}{{\rm d}}

\newcommand{\ri}{{\rm i}}

\newcommand{\cF}{\mathcal{F}}

\newcommand{\R}{\mathbb{R}}

\newcommand{\gl}{\mathfrak{\gl}}

\newcommand{\SO}{{\rm SO}}

\newcommand{\SU}{{\rm SU}}

\renewcommand{\epsilon}{\varepsilon}

\newcommand{\Ric}{{\rm Ric}}

\renewcommand{\Re}{\mathop{\mathrm{Re}}}

\newcommand{\tr}{\mathop{\mathrm{tr}}\nolimits}

\newcommand{\vol}{\mathrm{vol}}

\newcommand{\Rm}{\mathrm{Rm}}
\newcommand{\scal}{\mathrm{Scal}}

\newcommand{\gt}{\texorpdfstring{\mathrm{G}_2}{\space}}

\newcommand{\qandq}{\quad\text{and}\quad}

\def\<{\mathopen{}\left<}
\def\>{\right>\mathclose{}}
\def\({\mathopen{}\left(}
\def\){\right)\mathclose{}}

\usepackage{multicol, color}

\definecolor{gold}{rgb}{0.85,.66,0}
\definecolor{cherry}{rgb}{0.9,.1,.2}
\definecolor{burgundy}{rgb}{0.8,.2,.2}
\definecolor{orangered}{rgb}{0.85,.3,0}
\definecolor{orange}{rgb}{0.85,.4,0}
\definecolor{olive}{rgb}{.45,.4,0}
\definecolor{lime}{rgb}{.6,.9,0}
\definecolor{green}{rgb}{.2,.7,0}
\definecolor{grey}{rgb}{.4,.4,.2}
\definecolor{brown}{rgb}{.4,.3,.1}

\newtheorem{theorem}{Theorem}
\newtheorem{prop}{Proposition}
\newtheorem{cor}{Corollary}

\numberwithin{substep}{step}

\numberwithin{subcase}{case}

\theoremstyle{remark}
\newtheorem{remark}{Remark}
\newtheorem{observation}{Observation}

\theoremstyle{definition}
\newtheorem{definition}{Definition}

\DeclareMathOperator{\rivarphi}{\ri_{\varphi}}

\DeclareMathOperator{\Gl}{Gl}

\usepackage{amssymb}

	\title[\bf Geometric Reductions of the $G_2$-Hilbert Functional via Circle Actions] %
{ Geometric Reductions of the $G_2$-Hilbert Functional via Circle Actions}  
\author[Julieth Saavedra]{}



\email{julieth.p.saavedra@gmail.com}

\thanks{$^*$Corresponding author: Julieth Saavedra}

\begin{document}
	\maketitle

\begin{center}
        \begin{minipage}{5cm}
	\centerline{\scshape Julieth Saavedra $^*$}
	\medskip
	{\footnotesize
		\centerline{Escuela de Ciencias Físicas y Matemáticas}
		\centerline{Universidad de Las Américas}
		\centerline{V\'ia a Nay\'on, C.P.170124, Quito, Ecuador}
	} 
\end{minipage}
\end{center}	
	
\bigskip

\begin{abstract}
In this paper, we study critical points and gradient flows of the $G_2$--Hilbert functional on a manifolds with free $\mathbb S^1$--actions. We analyze $\mathbb S^1$--invariant $G_2$--structures under the constant fiber-length non-K\"ahler transverse ansatz, reducing the variational problem to the $6$--dimensional quotient and we also consider a Gibbons--Hawking-type ansatz with varying fiber length and derive the formal negative $L^2$--gradient flow. We conclude that the unnormalized flow admits only trivial stationary configurations: flat connection, scalar-flat base metric, and constant fiber length. 
\end{abstract}

\section{Introduction}
Geometric flows of $G_2$-structures have been studied by many authors, such as \cite{Bryant2011,Dwivedi2019,fino2017closed,Karigiannis2007,lauret2021search,lotay2022CCY,Lotay2017}. Most of these flows come from functionals that take a class of curves over the tangent space of $G_2$ structures. Examples of $G_2$-structure flows that connect with functionals are the Laplacian flow \cite{Bryant2011} and the isometric gradient flow \cite{Dwivedi2019}. The importance of a flow arising from a functional is that one expects to use tools similar to those introduced by Perelman for the Ricci flow \cite{perelman}.

One of the important properties that Perelman used to show that the Ricci flow can be viewed as a gradient flow was the fact that he defined a functional invariant under diffeomorphisms. Recall that the Ricci flow itself admits a variational interpretation, in an appropriate sense, on a space arising from Riemannian metrics \cite{perelman}. Taking this into account, in \cite{gianniotis2025} they attempted to define a geometric functional that satisfied that property in the context of $G_2$-structures. They called $G_2$-Hilbert functional and it is defined as 
\begin{equation}\label{Functional}
    \mathcal F(\varphi) = \int_M \left( \frac{1}{6}\scal(g_\varphi) - \frac{1}{3}|T|_{g_\varphi}^2 - \frac{1}{6}(\tr T)^2 \right) dV_{g_\varphi},
\end{equation}

where $g_\varphi$ is the metric induced by the $G_2$-structure $\varphi$, $T$ denotes its full torsion tensor and $\scal(g_{\varphi})$ is the scalar curvature of the Levi-Civita connection associated with $g_{\varphi}$. 

In this paper, we study the $G_2$-Hilbert functional for $G_2$-structures on a seven-dimensional spin orientable manifold $M$ endowed with a free $\mathbb{S}^1$-action. More precisely, we consider $\mathbb{S}^1$-invariant $G_2$-structures, namely those satisfying $\mathcal{L}_\xi\varphi=0$, where $\xi$ denotes the vector field generated by the $\mathbb{S}^1$-action.

In different contexts, an $\mathbb{S}^1$-action on a manifold admitting a $G_2$-structure has been used to analyze the behavior of families of $G_2$-structures satisfying certain geometric properties; for instance, families arising as solutions of geometric flows. This is the case, for example, for closed $G_2$-structures under the Laplacian flow \cite{fowdar}. Another example appears in the Kähler context, where manifolds with holonomy equal to $G_2$ and admitting an $\mathbb{S}^1$-action are studied in relation to quotient geometries \cite{apostolov2004kahler}.

Therefore, we take a free $\mathbb{S}^1$-action generated by a nowhere-vanishing vector field $\xi$ and let $
h=\|\xi\|_\varphi^{-2},
$ be a function determined by the $\mathbb{S}^1$-action and the metric $g_{\varphi}$, and let $
\eta=h\,g_\varphi(\xi,\cdot)
$ be a $1$-form on $M$. 
Then, in Proposition \ref{Prop.1}, we show that $h$ is a positive basic function, $\eta$ is an $\mathbb{S}^1$-invariant connection $1$-form satisfying $
\eta(\xi)=1,
$ and $\varphi$ determines an induced $SU(3)$-structure $(\omega,\Omega)$ on the quotient $
N=M/\mathbb{S}^1.
$
The $\SU(3)$--structure $(\omega, \Omega=\Omega_+ + i\Omega_-)$ on the quotient $N=M/S^1$ is defined by the basic forms:
\begin{equation}\label{Eq.SU(3)}
    \begin{split}
        \pi^*\omega &= \iota_\xi\varphi, \\
        \pi^*\Omega_+ &= h^{-3/4}(\varphi-\eta\wedge\iota_\xi\varphi), \\
        \pi^*\Omega_- &= -h^{-1/4}\iota_\xi(*_\varphi\varphi).
    \end{split}
\end{equation}

Thus, the $G_2$--structure is given by 
\begin{equation}
    \begin{split}
        \varphi &= \eta\wedge \pi^*\omega + h^{3/4}\pi^*\Omega_+, \\
        *_\varphi\varphi &= \frac12\,h\,\pi^*(\omega^2)-h^{1/4}\eta\wedge \pi^*\Omega_-, \\
        g_\varphi &= h^{-1}\eta^2 + h^{1/2}\pi^*g_{(\omega,\Omega)}.
    \end{split}
\end{equation}

We are interested in considering two ansätze for $\mathbb{S}^1$-invariant $G_2$-structures on a manifold $M$ endowed with a free $\mathbb{S}^1$-action. The first one is given by
$h=1$ and $d\eta=\pi^*F$, where
\[
F=F_0^{1,1}+\lambda\,\omega,
\]
with $F_0^{1,1}$ the primitive real $(1,1)$-component and $\lambda$ a basic function, and where the transverse $SU(3)$-structure satisfies
\[
d\Omega=\theta\wedge\Omega,
\qquad
d\omega=\frac{2}{3}\theta\wedge\omega+\nu_3.
\]
where $\theta$ is a basic form and $\nu_3$ is the primitive torsion component of type $W_3$ and the second case is the Gibbons--Hawking-type ansatz, where one assumes that $\varphi$ is closed. For the first case, we can write the functional reduces as
\begin{equation}
    \mathcal F(\varphi)
=
-2\pi\int_N
\left(
\frac{15}{8}\lambda^2
+\frac{5}{18}|\theta|_{g_N}^2
+\frac14|\nu_3|_{g_N}^2
+\frac14|F_0^{1,1}|_{g_N}^2
\right)dV_{g_N}.
\end{equation}
In Theorem \ref{Th.Grad.flow_W345}, we obtain the associated gradient flow and solve part of the system. Moreover, in Corollary \ref{cor.critical.points.nonKahler}, we show that the only formal critical point of the reduced $G_2$--Hilbert functional is given by
\[
\lambda=0,
\qquad
\theta=0,
\qquad
\nu_3=0,
\qquad
F_0^{1,1}=0.
\]
Equivalently,
\[
d\eta=0,
\qquad
d\omega=0,
\qquad
d\Omega=0.
\]
Thus, the formal critical point corresponds to a flat connection over a transverse Calabi--Yau $SU(3)$-structure.

For the Gibbons--Hawking-type ansatz, we conclude that the $G_2$-Hilbert functional is given by
$$   \mathcal{F}(\varphi)
=
2\pi\int_N
\left(
\frac12 h^{1/2}\,\scal(g_\omega)
-\frac18 h^{-1}|d\eta|_{g_\omega}^2
\right)dV_{g_\omega}.$$
As we show in Theorem \ref{Th.gradientGH}, the resulting flow is a highly coupled, nonlinear system of evolution equations that simultaneously deforms the metric $g$ of the base, the fiber-length function $h$, and the connection $\eta$ on the principal bundle. Moreover, Corollary \ref{cor:flow_consequences} shows that the de Rham class of the curvature form $F=d\eta$ is preserved along the flow, while any stationary limit satisfies
$$
d^*(h^{-1}F)=0,
\qquad
\scal(g)=-\frac12 h^{-3/2}|F|^2\leq 0.
$$
In particular, Proposition \ref{prop.rig.gh} shows that, on a closed base, the unnormalized flow is rigid at stationary points and the only stationary configurations are those with
$$
F\equiv 0,
\qquad
\scal(g)=0,
\qquad
h=\mathrm{constant}.
$$
Thus, nontrivial stationary limits cannot occur for the unnormalized flow in this class.

\section{Preliminaries}

\subsection{$\SU(3)$-Structures}
\begin{definition}
An $\SU(3)$-structure on a $6$-manifold $N$ is a pair of smooth differential forms $(\omega,\Omega)$, where $\omega$ is a real non-degenerate $2$-form and $\Omega$ is a complex volume form, satisfying
\begin{equation}
    \omega\wedge\Re\Omega=0,\quad \frac{1}{6}\omega^3=\frac{1}{4}\text{Re}\Omega\wedge\text{Im}\Omega 
\end{equation}
\end{definition}
In \cite{Hitchin2000}, Hitchin shows that $\Omega_+$, or $\Omega_-$, determines the almost complex structure $J$, then the metric is given by $g_{\omega}(\cdot,\cdot)=\omega(\cdot,J\cdot)$.  As $S U(3)$ modules the space of exterior differential forms splits as (See \cite{Chiossi2001})

$$
\begin{aligned}
& \Lambda^1=\Lambda_6^1=[[\Lambda^{1,0}]] \\
& \Lambda^2=\langle\omega\rangle \oplus \Lambda_6^2 \oplus \Lambda_8^2 \\
& \Lambda^3=\left\langle\Omega_{+}\right\rangle \oplus\left\langle\Omega_{-}\right\rangle \oplus \Lambda_6^3 \oplus \Lambda_{12}^3
\end{aligned}
$$

where

$$
\begin{aligned}
\Lambda_6^2=[[\Lambda^{2,0}]]  & =\left\{\alpha \in \Lambda^2 \mid *_\omega(\alpha \wedge \omega)=\alpha\right\} \\
& =\left\{*_\omega\left(\alpha \wedge \Omega^{+}\right) \mid \alpha \in \Lambda_6^1\right\} \\
\Lambda_8^2=\left[\Lambda_0^{1,1}\right] & =\left\{\alpha \in \Lambda^2 \mid *_\omega(\alpha \wedge \omega)=-\alpha\right\} \\
& =\left\{\alpha \in \Lambda^2 \mid \alpha \wedge \omega^2=0 \text { and } \alpha \wedge \Omega^{+}=0\right\} \\
\Lambda_6^3=[[\Lambda^{2,1} ]]& =\left\{\alpha \wedge \omega \mid \alpha \in \Lambda_6^1\right\} \\
\Lambda_{12}^3=[[\Lambda_0^{2,1} ]]& =\left\{\alpha \in \Lambda^3 \mid \alpha \wedge \omega=0, \alpha \wedge \Omega^{ \pm}=0\right\}
\end{aligned}
$$

The intrinsic torsion of the $\SU(3)$-structure ( $\omega, \Omega^{+}$) is determined by

$$
\begin{aligned}
d \omega & =-\frac{3}{2} \sigma_0 \Omega^{+}+\frac{3}{2} \pi_0 \Omega^{-}+\nu_1 \wedge \omega+\nu_3 \\
d \Omega^{+} & =\pi_0 \omega^2+\pi_1 \wedge \Omega^{+}-\pi_2 \wedge \omega \\
d \Omega^{-} & =\sigma_0 \omega^2+\left(J \pi_1\right) \wedge \Omega^{+}-\sigma_2 \wedge \omega
\end{aligned}
$$

where $\sigma_0, \pi_0$ are functions, $\nu_1, \pi_1 \in \Lambda_6^1, \pi_2, \sigma_2 \in \Lambda_8^2$ and $\nu_3 \in \Lambda_{12}^3$, cf. \cite{Chiossi2001}.

\subsection{$G_2$-structure}

Let $\{e_1,\cdots,e_7\}$ be the standard basis of $\R^7$ and we will denote its dual by $\{e^1,\cdots,e^7\}$. We define a $3$-form $\phi$ by 
\begin{equation}
    \phi=e^{123}+e^{145}+e^{167}+e^{246}-e^{257}-e^{347}-e^{356}
\end{equation}
where $e^{ijk}=e^i\wedge e^j\wedge e^k$. The subgroup of $\Gl(7,\R)$ fixing $\phi$ is the Lie group $G_2$, which is compact, connected, simple lie subgroup of $\SO(7)$ of dimension $14$. We also have that $G_2$ preserve the $4$-form 
\begin{equation}
    \ast_{\phi}\phi=e^{4567}+e^{2367}+e^{2345}+e^{1357}-e^{1346}-e^{1256}-e^{1247}.
\end{equation}
where $\ast_\phi$ is the Hodge star operator determined by the canonical metric and orientation of $\R^7$. Let $M$ be a be a smooth $7$-manifold which is oriented and spin . For $x\in M$ we let
\begin{equation*}
    \Lambda^3_{+}(M)_x=\{\varphi_x\in\Lambda^3T^{\ast}M|\exists u\in \text{Hom}(T_xM,\R^7),u^{\ast}\phi=\varphi_x\}
\end{equation*}
and thus, we obtain a bundle $\Lambda^3_+(M)=\sqcup_x\Lambda^3_+(M)_x$ which is an open subbundle of $\Lambda^3T^{\ast}M$, with fiber $\Gl(7,\R)/G_2$. We call a positive $3$-form $\varphi$ a section of the bundle $\Lambda^3_+(M)$ and we will denote by $\Omega^3_+(M)=\Gamma(\Lambda^3_+(M))$. For a positive $3$-form $\varphi$, we can associate a bilinear form $B_{\varphi}(u,v)=\frac{1}{6}(u\lrcorner \varphi)\wedge(v\lrcorner\varphi)\wedge\varphi$
where $u,v$ are tangent vectors on $M$. Then, we can see that $\varphi\in\Omega^3_+(M)$ determines a Riemannian metric $g_{\varphi}$ and an orientation $dV_{\varphi}$, hence the Hodge star operator $\ast_{\varphi}$ and the associated $4$-form $\psi:=\ast_{\varphi}\varphi$. A $\gt$-structure gives rise to a decomposition of the space of differential $k$-forms $\Omega^k$ on $M$ into irreducible $\gt$-submodules. For instance,
\begin{align*}
	\Omega^2 &= \Omega^2_7\oplus\Omega^2_{14}\qandq
	\Omega^3  = \Omega^3_1\oplus\Omega_{7}^{3}\oplus\Omega^3_{27},
\end{align*}
where $\Omega^k_l$ has (pointwise) dimension $l$.According with the $\gt$-decomposition of $\Omega^4$ and $\Omega^5$, the exterior derivative of $\varphi$ and $\psi$ are completely described in term of the \emph{torsion forms} $\tau_0\in\Omega^0$, $\tau_1\in\Omega^1$,  $\tau_2\in\Omega_{14}^2$ and $\tau_3\in\Omega^3_{27}$,  given in terms of (see ~\cite[Proposition 1]{Bryant2006})
\begin{align}\label{eq: Fernandez dpsi}
\begin{split}
  	\rd\varphi 
  	=&\tau_0\psi+ 3\tau_1\wedge\varphi+\ast\tau_3  \in  \Omega_1^4\oplus\Omega^4_7\oplus\Omega^4_{27}\\ 
	\rd\psi 
=&4\tau_1\wedge\psi+\tau_2\wedge\varphi \in  \Omega^5_7\oplus\Omega^5_{14}.  
\end{split}
\end{align}
Moreover, for the \emph{full torsion tensor} is defined locally by (see \cite{Karigiannis2007})
\begin{equation}
\label{Eq.nabla.varphi}
	\nabla_i\varphi_{jkl}
	=T_i^m\psi_{mjkl}.
\end{equation}
In addition, from \eqref{Eq.nabla.varphi} for the $4$-form $\psi$, we have
\begin{equation*}
    \nabla_m\psi_{ijkl}=-(T_{mi}\varphi_{jkl}-T_{mj}\varphi_{ikl}-T_{mk}\varphi_{jil}-T_{ml}\varphi_{jki}).
\end{equation*}
Thus, the full torsion tensor $T$ is given in terms of the torsion forms by
\begin{equation}
\label{Eq:Torsion}
	T =\frac{\tau_0}{4}g -\tau_{27}-\tau_1^{\sharp}\lrcorner\varphi -\frac{1}{2}\tau_2 ,
\end{equation}
where $\tau_{27}$ is the trace-free symmetric $2$-tensor satisfying $\tau_3=\rivarphi(\tau_{27})$ and $\tau_1^\sharp$ denotes the unique vector field induced by $\tau_1$ and the Riemannian metric $g$, (i.e. $g(\tau_1^\sharp,X)=\tau_1(X)$ for any $\in \mathfrak{X}(M)$) and its trace and norm are related to the torsion forms by
\begin{equation}\label{eq:Trelations}
\operatorname{tr}(T)=\frac{7}{4}\,\tau_0,
\qquad
|T|^2=\frac{7}{16}\tau_0^2
+24|\tau_1|^2
+\frac12|\tau_2|^2
+\frac12|\tau_3|^2.
\end{equation}

A $G_2$--structure $\varphi$ determines uniquely a Riemannian metric $g_\varphi$ on $M$, and hence its Levi--Civita connection $\nabla^{g_\varphi}$. The corresponding Riemann curvature tensor is the $(3,1)$--tensor defined by
\begin{equation}
\Rm(X,Y)Z
:=
\nabla^{g_\varphi}_X\nabla^{g_\varphi}_Y Z
-
\nabla^{g_\varphi}_Y\nabla^{g_\varphi}_X Z
-
\nabla^{g_\varphi}_{[X,Y]}Z,
\end{equation}
for all vector fields $X,Y,Z\in \mathfrak X(M)$. In a local coordinate chart $\{x^1,\dots,x^7\}$ on $M$, we write
\begin{equation}
\Rm\!\left(\frac{\partial}{\partial x^i},\frac{\partial}{\partial x^j}\right)\frac{\partial}{\partial x^k}
=
R_{ijk}{}^{l}\frac{\partial}{\partial x^l},
\qquad
R_{ijkl}=g_{lm}R_{ijk}{}^{m}.
\end{equation}
Thus, the curvature of $g_\varphi$ may be described locally in coordinates through its components $R_{ijk}{}^l$.

On the other hand, from the point of view of $G_2$--geometry, the scalar curvature admits an intrinsic expression in terms of the torsion forms of $\varphi$. More precisely, one has (see \cite{Bryant2006})
\begin{equation}\label{eq.Scal.curvature}
\scal(g_\varphi)
=
12\,d^*\tau_1
+\frac{21}{8}\tau_0^2
+30|\tau_1|^2
-\frac12|\tau_2|^2
-\frac12|\tau_3|^2.
\end{equation}

\section{$S^1$-invariant action of $G_2$-structures}
 
It is well known that the Ricci flow can be interpreted, in an appropriate sense, as a gradient flow on the space of Riemannian metrics. One natural approach to this interpretation is through the Einstein--Hilbert functional. However, the presence of additional terms arising from diffeomorphism invariance complicates the identification of the Ricci flow as the pure gradient flow of this functional. In \cite{gianniotis2025}, the authors introduced a diffeomorphism-invariant $G_2$-Hilbert functional. Our aim is to analyze this $\mathcal{F}$-functional in the case of manifolds admitting a regular $S^1$-action.
\begin{prop}\label{Prop.1}
Let $M$ be a $7$-dimensional manifold endowed with a free $\mathbb{S}^1$-action generated by a nowhere-vanishing vector field $\xi$. Assume that $M$ carries an $\mathbb{S}^1$-invariant $G_2$-structure $\varphi$, that is, $
\mathcal L_\xi \varphi=0.
$ Let
$$
h=\|\xi\|_\varphi^{-2},
\qquad
\eta=h\,g_\varphi(\xi,\cdot).
$$
Then $h$ is a positive basic function, $\eta$ is an $\mathbb{S}^1$-invariant connection $1$-form satisfying $
\eta(\xi)=1,
$ and $\varphi$ determines an induced $SU(3)$-structure $(\omega,\Omega)$ on the quotient $
N=M/\mathbb{S}^1.
$ Hence the invariant $G_2$-structure $\varphi$ is equivalent to reduced geometric data
$
(h,\eta,\omega,\Omega)
$
on the quotient. 
\end{prop}

\begin{proof}
Since the $G_2$-structure $\varphi$ is $\mathbb{S}^1$-invariant ($\mathcal{L}_\xi \varphi = 0$), the induced Riemannian metric $g_\varphi$ is also invariant, meaning $\xi$ is a Killing vector field ($\mathcal{L}_\xi g_\varphi = 0$). Because the action is free, $\xi$ is nowhere vanishing, which implies $\|\xi\|_\varphi^2 > 0$, and thus $h > 0$ globally. The invariance of the metric implies $\mathcal{L}_\xi (\|\xi\|_\varphi^2) = 0$, so $dh(\xi) = 0$. Since $h$ is constant along the flow of $\xi$ and invariant under the group action, it is a basic function and descends to a well-defined positive smooth function on the quotient $N$. Evaluating the $1$-form $\eta$ on the fundamental vector field yields to 
$
\eta(\xi)=1.
$ Furthermore, the Lie derivative of $\eta$ along $\xi$ is:
$$
\mathcal{L}_\xi \eta = (\mathcal{L}_\xi h) g_\varphi(\xi, \cdot) + h (\mathcal{L}_\xi g_\varphi)(\xi, \cdot) + h\, g_\varphi(\mathcal{L}_\xi \xi, \cdot) = 0,
$$
since $\mathcal{L}_\xi \xi = [\xi, \xi] = 0$. Moreover
, $\eta$ is invariant and evaluates to $1$ on the fundamental vector field, it defines a valid principal connection $1$-form on the bundle $\pi: M \to N$.

This connection induces a splitting of the tangent bundle $TM = \mathcal{V} \oplus \mathcal{H}$, where $\mathcal{V} = \text{span}(\xi)$ is the vertical distribution and $\mathcal{H} = \ker(\eta)$ is the horizontal distribution. At each point, the stabilizer of a non-zero vector in $\mathbb{R}^7$ under the linear action of $G_2$ is isomorphic to $\mathrm{SU}(3)$. By fixing the normalized vector field $h^{1/2}\xi$, the $G_2$-structure orthogonally decomposes, restricting to an $\mathrm{SU}(3)$-structure on the $6$-dimensional horizontal space $\mathcal{H}$. Because $\mathcal{L}_\xi \varphi = 0$, this horizontal algebraic structure is constant along the fibers and descends via $d\pi$ to uniquely define an $\mathrm{SU}(3)$-structure $(\omega, \Omega)$ on $TN$. Consequently, the invariant data $\varphi$ on $M$ is in bijection with the reduced tuple $(h, \eta, \omega, \Omega)$ on $N$. 
\end{proof}
In particular, any geometric functional on the space of $S^1$--invariant $G_2$--structures reduces to one depending on $(h,\eta,\omega,\Omega)$, where $h=\|\xi\|_\varphi^{-2}$ and $\eta=h\,g_\varphi(\xi,\cdot)$. With this setup, $\eta(\xi)=1$ and the quotient $N=M/S^1$ carries a transverse $SU(3)$--structure $(\omega,\Omega=\Omega_+ + i\,\Omega_-)$ such that:
\begin{equation}\label{Eq.G2-ansatz}
\begin{split}
    \varphi =& \eta\wedge \pi^*\omega + h^{3/4}\pi^*\Omega_+,\\
    *_\varphi\varphi =& \frac12\,h\,\pi^*(\omega^2) - h^{1/4}\eta\wedge \pi^*\Omega_-,\\
    g_\varphi =& h^{-1}\eta^2 + h^{1/2}\pi^*g_{(\omega,\Omega)}.
    \end{split}
\end{equation}

\begin{prop}
Let $\varphi$ be an $S^1$--invariant $G_2$--structure on a $7$--manifold $M$ with a free $S^1$--action generated by $\xi$. Let $\pi:M\longrightarrow N=M/S^1$ be the associated principal $S^1$--bundle. Defining $h=\|\xi\|_\varphi^{-2}$ and $\eta=h\,g_\varphi(\xi,\cdot)$, we have $\eta(\xi)=1$ and $\varphi$ takes the form
$$
\varphi = \eta\wedge \pi^*\omega + h^{3/4}\pi^*\Omega_+,
$$
where $(\omega,\Omega)$ is the induced $SU(3)$--structure on the quotient $N$.

Furthermore, the tangent space $T_\varphi\bigl(\Omega^3_+(M)^{S^1}\bigr)$ decomposes as
$$
C^\infty(N)\oplus \Omega^1(N)\oplus T_{(\omega,\Omega)}\mathcal{SU}(3),
$$
such that for any variation $(\dot h,\dot\eta,\dot\omega,\dot\Omega_+)$ in this space, the corresponding infinitesimal variation of $\varphi$ is given by
\begin{equation}\label{eq:variation-phi-S1-corrected}
\dot\varphi = \dot\eta\wedge \pi^*\omega + \eta\wedge \pi^*\dot\omega + \frac34 h^{-1/4}\dot h\,\pi^*\Omega_+ + h^{3/4}\pi^*\dot\Omega_+.
\end{equation}
\end{prop}

\begin{proof}
From $h=\|\xi\|_\varphi^{-2}$ and $\eta=h\,g_\varphi(\xi,\cdot)$, it directly follows that $\eta(\xi) = h\|\xi\|_\varphi^2 = 1$. Thus, $\eta$ defines a connection one-form on the principal $S^1$--bundle $\pi:M\longrightarrow N$.

Since $\mathcal L_\xi\varphi=0$, the horizontal components of $\varphi$ are basic and descend to the quotient $N$. The $S^1$--invariant positive $3$--form $\varphi$ is therefore entirely determined by the quadruple $(h,\eta,\omega,\Omega)$ via the ansatz
$$
\varphi = \eta\wedge\pi^*\omega + h^{3/4}\pi^*\Omega_+.
$$

Considering a smooth one-parameter family of $S^1$--invariant positive $3$--forms $\varphi_t$ with $\varphi_0=\varphi$, its infinitesimal variation is determined by the variations of the induced data: a function $\dot h \in C^\infty(N)$, a basic one-form $\dot\eta \in \Omega^1(N)$, and a transverse $SU(3)$--structure variation $(\dot\omega,\dot\Omega) \in T_{(\omega,\Omega)}\mathcal{SU}(3)$. This establishes the isomorphism
$$
T_\varphi\bigl(\Omega^3_+(M)^{S^1}\bigr) \cong C^\infty(N)\oplus\Omega^1(N)\oplus T_{(\omega,\Omega)}\mathcal{SU}(3).
$$

Finally, differentiating the ansatz for $\varphi_t$ at $t=0$ immediately yields the corresponding formula for the variation:
$$
\dot\varphi = \dot\eta\wedge \pi^*\omega + \eta\wedge \pi^*\dot\omega + \frac34 h^{-1/4}\dot h\,\pi^*\Omega_+ + h^{3/4}\pi^*\dot\Omega_+.
$$
\end{proof}

\subsection{The constant fiber-length $W_3\oplus W_4\oplus W_5$ transverse ansatz}
Let $M$ be a manifold equipped with a free $S^1$-action and an $S^1$-invariant $G_2$-structure $\varphi$. Let $\xi$ denote the fundamental vector field of the action and assume the fiber-length function is constant, $h=1$. 
To work with non-K\"ahler transverse $SU(3)$-structures, we work with a more general condition
\begin{equation}\label{Eq.curvature-general}
d\eta=\pi^*F,
\qquad
dF=0,
\end{equation}
where $F$ is a basic closed $2$--form on the quotient
$N=M/\mathbb S^1.
$ With respect to the transverse $SU(3)$-structure $(\omega,\Omega)$ on $N$, the form $F$ decomposes as
\begin{equation}\label{Eq.F-decomposition}
F=F_0^{1,1}+\lambda\,\omega+F^{2,0+0,2},
\end{equation}
where $F_0^{1,1}$ denotes the primitive real $(1,1)$--component, $\lambda\,\omega$ is the trace component, with $\lambda$ a basic function on $N$, and $F^{2,0+0,2}$ is the real component of type $(2,0)+(0,2)$.

In the present work, we restrict our attention to the $(1,1)$--curvature case by imposing
\[
F^{2,0+0,2}=0.
\]
Thus,
\begin{equation}\label{Eq.F-11}
F=F_0^{1,1}+\lambda\,\omega
\end{equation}
Since $F$ is closed, we obtain
\begin{equation}\label{Eq.closed-F}
dF_0^{1,1}+d\lambda\wedge\omega+\lambda\,d\omega=0.
\end{equation}
We impose the following structure equations for the transverse geometry:
\begin{equation}\label{Eq.transverse-structure}
d\Omega=\theta\wedge\Omega,
\qquad
d\omega=\frac{2}{3}\theta\wedge\omega+\nu_3,
\end{equation}
where $\theta$ is a basic real $1$--form on $N$, and $\nu_3$ is the primitive torsion component of type $W_3$ of the transverse $SU(3)$-structure. In particular,
\[
\nu_3\in\Omega^3(N),
\qquad
\nu_3\wedge\omega=0.
\]
Whenever these equations are used on the total space $M$, all forms are understood as pulled back by $\pi$.
Substituting \eqref{Eq.transverse-structure} into \eqref{Eq.closed-F}, we obtain the compatibility condition
\begin{equation}\label{Eq.Compatibility}
dF_0^{1,1}
+
\left(
d\lambda+\frac{2}{3}\lambda\,\theta
\right)\wedge\omega
+
\lambda\,\nu_3
=0.
\end{equation}
This identity reflects the interaction between the curvature of the connection and the torsion of the transverse $SU(3)$-structure. Several cases are worth noting:
\begin{itemize}
    \item If $\lambda$ is a nonzero constant and $F_0^{1,1}$ is closed, then \eqref{Eq.Compatibility} forces
   $d\omega=0.
    $ Hence the transverse geometry reduces to the K\"ahler case, and in the present torsion ansatz this implies
   $ \theta=0,
    $ and $
    \nu_3=0.
    $

    \item If $\lambda=0$, then
    $F=F_0^{1,1},
    $ and the condition $dF=0$ reduces to
   $ dF_0^{1,1}=0.
    $ In this case, the curvature is purely primitive and does not force the transverse structure to be K\"ahler.

    \item If $\lambda$ is a nonconstant basic function, then the term
  $ d\lambda\wedge\omega
    $ can compensate for the non-closedness of $\omega$, allowing nontrivial Lee form $\theta$ and non-K\"ahler transverse torsion.
\end{itemize}

Consequently, the structural ansatz considered in this work is
\begin{equation}\label{Eq.structural-ansatz}
\begin{cases}
d(\pi^*\Omega)=\pi^*\theta\wedge\pi^*\Omega,\\[0.4em]
d(\pi^*\omega)=\dfrac{2}{3}\pi^*\theta\wedge\pi^*\omega+\pi^*\nu_3,\\[0.6em]
d\eta=\pi^*F=\pi^*F_0^{1,1}+\pi^*(\lambda\,\omega),
\end{cases}
\end{equation}
subject to the compatibility condition \eqref{Eq.Compatibility}. This framework keeps the curvature compatible with the transverse complex structure while allowing non-K\"ahler transverse torsion.

\begin{observation}
Throughout this work, in the constant fiber-length non-K\"ahler quotient ansatz setting, the forms $\omega$, $\Omega$, $\Omega_+$, $\Omega_-$, $\theta$, $\nu_3$, $F$, and $F_0^{1,1}$ are defined on the quotient manifold $N$. When they appear in formulas on the total space $M$, they should be understood as their pullbacks by $\pi$.
\end{observation}
\begin{prop}\label{prop.bwansatz}
Let $\varphi$ be an $\mathbb S^1$-invariant $G_2$-structure on $M$ defined by the constant fiber-length ansatz
\begin{equation}
    \varphi = \eta \wedge \omega + \Omega_+,
    \qquad
    *_\varphi\varphi = \frac{1}{2}\omega^2 - \eta \wedge \Omega_-.
\end{equation}
Assume also the ansatz given by \eqref{Eq.structural-ansatz}. Then the torsion forms of $\varphi$ are given by
\begin{equation}
    \begin{split}
        \tau_0 &= \frac{6\lambda}{7}, \\
        \tau_1 &= \frac{5}{18}\,\theta, \\
        \tau_2 &= \frac{2}{9}\,\eta\wedge J\theta
        + \frac{1}{9}\,\iota_{\theta^\sharp}\Omega_+, \\
        \tau_3 &=
        \frac{8\lambda}{7}\,\eta\wedge\omega
        -\frac{6\lambda}{7}\,\Omega_+
        +\frac{1}{6}\,J\theta\wedge\omega
        -\frac{1}{6}\,\eta\wedge\iota_{\theta^\sharp}\Omega_-
        -*_6\nu_3
        -\eta\wedge F_0^{1,1}.
    \end{split}
\end{equation}
Furthermore, the trace and squared norm of the intrinsic torsion tensor $T_\varphi$, together with the scalar curvature $\mathrm{Scal}(g_\varphi)$, satisfy
\begin{align}
    \operatorname{Tr}(T_\varphi)
    &=
    \frac{3}{2}\lambda,
    \\
    |T_\varphi|^2
    &=
    \frac{15}{4}\lambda^2
    +
    \frac{35}{18}|\theta|^2
    +
    \frac{1}{2}|\nu_3|^2
    +
    \frac{1}{2}|F_0^{1,1}|^2,
    \\
    \mathrm{Scal}(g_\varphi)
    &=
    \frac{10}{3}\,d^*\theta
    -
    \frac{3}{2}\lambda^2
    +
    \frac{20}{9}|\theta|^2
    -
    \frac{1}{2}|\nu_3|^2
    -
    \frac{1}{2}|F_0^{1,1}|^2.
\end{align}
Here $J$ is the almost complex structure determined by $(\omega,\Omega)$, $\theta^\sharp$ is the metric dual of $\theta$ with respect to $g_{(\omega,\Omega)}$, and $*_6$ denotes the Hodge star operator on the quotient.
\end{prop}
\begin{proof}
We use the standard decomposition of the torsion forms of a $G_2$-structure \eqref{eq: Fernandez dpsi}
and under the constant fiber-length ansatz, we have
\[
\varphi=\eta\wedge\omega+\Omega_+,
\qquad
*_\varphi\varphi=\frac{1}{2}\omega^2-\eta\wedge\Omega_-.
\]
Using the structure equations \eqref{Eq.structural-ansatz}, we obtain
\begin{align*}
d\varphi
&=
d\eta\wedge\omega-\eta\wedge d\omega+d\Omega_+ \\
&=
\lambda\omega^2
+
F_0^{1,1}\wedge\omega
-\frac{2}{3}\eta\wedge\theta\wedge\omega
-\eta\wedge\nu_3
+\theta\wedge\Omega_+,
\end{align*}
and
\begin{align*}
d(*_\varphi\varphi)
&=
d\left(\frac{1}{2}\omega^2-\eta\wedge\Omega_-\right)=
d\omega\wedge\omega-d\eta\wedge\Omega_-+\eta\wedge d\Omega_-.
\end{align*}
Since $\nu_3$ is primitive, $\nu_3\wedge\omega=0$. Moreover, since $F_0^{1,1}+\lambda\omega$ is of type $(1,1)$, we have
\[
(F_0^{1,1}+\lambda\omega)\wedge\Omega_-=0.
\]
Therefore
\[
d(*_\varphi\varphi)
=
\frac{2}{3}\theta\wedge\omega^2
+
\eta\wedge\theta\wedge\Omega_-.
\]
Using the standard $SU(3)$ identities (see \cite[Equation 13]{moroianu2008deformations})
\[
*_6F_0^{1,1}=-F_0^{1,1}\wedge\omega,
\qquad
\iota_{\theta^\sharp}\Omega_+\wedge\omega=-J\theta\wedge\Omega_-,
\]
together with the corresponding $G_2$ Hodge-star identities, we obtain
\begin{gather*}
\tau_0=\frac{6\lambda}{7},
\qquad
\tau_1=\frac{5}{18}\theta,
\qquad
\tau_2=
\frac{2}{9}\eta\wedge J\theta
+
\frac{1}{9}\iota_{\theta^\sharp}\Omega_+,
\\
\tau_3
=
\frac{8\lambda}{7}\eta\wedge\omega
-\frac{6\lambda}{7}\Omega_+
+\frac{1}{6}J\theta\wedge\omega
-\frac{1}{6}\eta\wedge\iota_{\theta^\sharp}\Omega_-
-*_6\nu_3
-\eta\wedge F_0^{1,1}.
\end{gather*}
The term $-\eta\wedge F_0^{1,1}$ contributes precisely the additional term
$F_0^{1,1}\wedge\omega
$ in $d\varphi$, because
\[
*_\varphi(-\eta\wedge F_0^{1,1})
=
-*_6F_0^{1,1}
=
F_0^{1,1}\wedge\omega.
\]
The norms of the torsion forms are
\begin{gather*}
|\tau_0|^2=\frac{36}{49}\lambda^2,
\qquad
|\tau_1|^2=\frac{25}{324}|\theta|^2,
\qquad
|\tau_2|^2=\frac{2}{27}|\theta|^2,\\
|\tau_3|^2
=
\frac{48}{7}\lambda^2
+
\frac{1}{9}|\theta|^2
+
|\nu_3|^2
+
|F_0^{1,1}|^2.
\end{gather*}
Here the cross terms vanish because the summands belong to mutually orthogonal $SU(3)$-irreducible components. Now, we wil use \eqref{eq:Trelations} for conclude
\begin{align}
\operatorname{Tr}(T_\varphi)
=&
\frac{7}{4}\cdot\frac{6\lambda}{7}
=
\frac{3}{2}\lambda.\\
|T_\varphi|^2
=&
\frac{15}{4}\lambda^2
+
\frac{35}{18}|\theta|^2
+
\frac{1}{2}|\nu_3|^2
+
\frac{1}{2}|F_0^{1,1}|^2.
\end{align}
Finally, using the scalar curvature formula \eqref{eq.Scal.curvature}, we obtain
\[
\mathrm{Scal}(g_\varphi)
=
\frac{10}{3}d^*\theta
-\frac{3}{2}\lambda^2
+\frac{20}{9}|\theta|^2
-\frac{1}{2}|\nu_3|^2
-\frac{1}{2}|F_0^{1,1}|^2.
\]
This proves the proposition.
\end{proof}
We now turn to the study of the $G_2$-Hilbert functional \eqref{Functional} restricted to curves of $G_2$-structures satisfying the constant fiber-length non-K\"ahler transverse ansatz. Under the assumptions of Proposition \ref{prop.bwansatz}, the $G_2$-Hilbert functional for $\mathbb S^1$-invariant $G_2$-structures is given by
\[
\mathcal F(\varphi)
=
2\pi\int_N
\left(
\frac{5}{9}\,d_N^*\theta
-\frac{15}{8}\lambda^2
-\frac{5}{18}|\theta|_{g_N}^2
-\frac14|\nu_3|_{g_N}^2
-\frac14|F_0^{1,1}|_{g_N}^2
\right)dV_{g_N}.
\]
Finally, if $N$ is compact without boundary, then
\[
\int_N d_N^*\theta\, dV_{g_N}=0,
\]
and therefore
\begin{equation}\label{eq.funct.theta}
\mathcal F(\varphi)
=
-2\pi\int_N
\left(
\frac{15}{8}\lambda^2
+\frac{5}{18}|\theta|_{g_N}^2
+\frac14|\nu_3|_{g_N}^2
+\frac14|F_0^{1,1}|_{g_N}^2
\right)dV_{g_N}.
\end{equation}

\begin{theorem}\label{prop.first variation.BW}
Let
$\varphi_t=\eta_t\wedge\omega_t+(\Omega_t)_+
$ be a smooth $1$--parameter family of $\mathbb S^1$--invariant $G_2$--structures satisfying the constant fiber-length non-K\"ahler transverse ansatz
\[
d\eta_t=F_t,
\qquad
F_t=(F_0^{1,1})_t+\lambda_t\omega_t,
\]
and
\[
d\Omega_t=\theta_t\wedge\Omega_t,
\qquad
d\omega_t=\frac{2}{3}\theta_t\wedge\omega_t+\nu_{3,t}.
\]
Assume that $N=M/\mathbb S^1$ is compact without boundary. Let $g_t:=g_{(\omega_t,\Omega_t)}.
$
Then the first variation of the reduced $G_2$--Hilbert functional \eqref{eq.funct.theta}, with respect to the variations
\[
k=\dot g_t\big|_{t=0},
\qquad
\dot\lambda=\frac{d}{dt}\Big|_{t=0}\lambda_t,
\qquad
\dot\theta=\frac{d}{dt}\Big|_{t=0}\theta_t,\quad
\dot\nu_3=\frac{d}{dt}\Big|_{t=0}\nu_{3,t},
\qquad
\dot F_0^{1,1}=\frac{d}{dt}\Big|_{t=0}(F_0^{1,1})_t,
\]
is given by
\begin{equation}\label{eq:first-variation-F-reduced}
\small
\begin{split}
\frac{d}{dt}\Big|_{0}\mathcal F(\varphi_t)
=
-2\pi\int_N
\Bigg[
&\frac{15}{4}\lambda\,\dot\lambda
+\frac{5}{9}\langle\theta,\dot\theta\rangle
+\frac12\langle\nu_3,\dot\nu_3\rangle
+\frac12\langle F_0^{1,1},\dot F_0^{1,1}\rangle
\\
&+
\left\langle
\left(
\frac{15}{16}\lambda^2
+\frac{5}{36}|\theta|^2
+\frac18|\nu_3|^2
+\frac18|F_0^{1,1}|^2
\right)g_N
-\frac{5}{18}\theta\otimes\theta
-\frac14(\nu_3\circ\nu_3)
-\frac14(F_0^{1,1}\circ F_0^{1,1}),
\,k
\right\rangle
\Bigg]dV_{g_N}.
\end{split}
\end{equation}
Here
$(\nu_3\circ\nu_3)_{ij}
=
\frac12(\nu_3)_{iab}(\nu_3)_{j}{}^{ab},
$
and
$(F_0^{1,1}\circ F_0^{1,1})_{ij}
=
(F_0^{1,1})_{ia}(F_0^{1,1})_{j}{}^{a}.
$

In addition, differentiating the closedness condition $dF_t=0$ gives
\[
d\dot F_0^{1,1}
+
d\dot\lambda\wedge\omega
+
d\lambda\wedge\dot\omega
+
\dot\lambda\,d\omega
+
\lambda\,d\dot\omega
=
0.
\]
\end{theorem}

\begin{proof}
Since $N$ is compact without boundary, the reduced functional is given by \eqref{eq.funct.theta}. 
Let $A$ a function defined as
\[
A
=
\frac{15}{8}\lambda^2
+
\frac{5}{18}|\theta|^2
+
\frac14|\nu_3|^2
+
\frac14|F_0^{1,1}|^2.
\]
We use the standard variation of the volume form
$\frac{d}{dt}\Big|_0 dV_{g_t}
=
\frac12\operatorname{tr}_g(k)dV_g.
$
Moreover,
\begin{align*}
\frac{d}{dt}\Big|_0|\theta_t|_{g_t}^2
=&
2\langle\theta,\dot\theta\rangle
-
\langle\theta\otimes\theta,k\rangle,\\
\frac{d}{dt}\Big|_0|\nu_{3,t}|_{g_t}^2
=&
2\langle\nu_3,\dot\nu_3\rangle
-
\langle\nu_3\circ\nu_3,k\rangle,\\
\frac{d}{dt}\Big|_0|(F_0^{1,1})_t|_{g_t}^2
=&
2\langle F_0^{1,1},\dot F_0^{1,1}\rangle
-
\langle F_0^{1,1}\circ F_0^{1,1},k\rangle.
\end{align*}
Therefore
\[
\begin{split}
\dot A
=
&\frac{15}{4}\lambda\,\dot\lambda
+\frac{5}{9}\langle\theta,\dot\theta\rangle
+\frac12\langle\nu_3,\dot\nu_3\rangle
+\frac12\langle F_0^{1,1},\dot F_0^{1,1}\rangle-\frac{5}{18}\langle\theta\otimes\theta,k\rangle
-\frac14\langle\nu_3\circ\nu_3,k\rangle
-\frac14\langle F_0^{1,1}\circ F_0^{1,1},k\rangle.
\end{split}
\]
Thus
\[
\frac{d}{dt}\Big|_0\mathcal F(\varphi_t)
=
-2\pi\int_N
\left(
\dot A+\frac12A\operatorname{tr}_g(k)
\right)dV_g.
\]
Since
\[
\frac12A\operatorname{tr}_g(k)
=
\left\langle \frac12A\,g_N,k\right\rangle,
\]
we obtain
\[
\begin{split}
\frac{d}{dt}\Big|_0\mathcal F(\varphi_t)
=
-2\pi\int_N
\Bigg[
&\frac{15}{4}\lambda\,\dot\lambda
+\frac{5}{9}\langle\theta,\dot\theta\rangle
+\frac12\langle\nu_3,\dot\nu_3\rangle
+\frac12\langle F_0^{1,1},\dot F_0^{1,1}\rangle
\\
&+
\left\langle
\frac12A\,g_N
-\frac{5}{18}\theta\otimes\theta
-\frac14(\nu_3\circ\nu_3)
-\frac14(F_0^{1,1}\circ F_0^{1,1}),
k
\right\rangle
\Bigg]dV_{g_N}.
\end{split}
\]
Substituting the definition of $A$ gives
\[
\frac12A\,g_N
=
\left(
\frac{15}{16}\lambda^2
+\frac{5}{36}|\theta|^2
+\frac18|\nu_3|^2
+\frac18|F_0^{1,1}|^2
\right)g_N,
\]
which proves \eqref{eq:first-variation-F-reduced}.
This completes the proof.
\end{proof}

\begin{cor}\label{cor.critical.points.nonKahler}
Let $N$ be a compact manifold without boundary, and let $\varphi$ be an $\mathbb S^1$-invariant $G_2$-structure satisfying the constant fiber-length non-K\"ahler transverse ansatz. Formally, if the reduced variables
\[
(\lambda,\theta,\nu_3,F_0^{1,1})
\]
are varied independently, then the only critical point of the reduced $G_2$--Hilbert functional is given by
\[
\lambda=0,
\qquad
\theta=0,
\qquad
\nu_3=0,
\qquad
F_0^{1,1}=0.
\]
Equivalently,
\[
d\eta=0,
\qquad
d\omega=0,
\qquad
d\Omega=0.
\]
Thus, the formal critical point corresponds to a flat connection over a transverse Calabi--Yau $SU(3)$-structure.
\end{cor}

The constant fiber-length non-K\"{a}hler transverse ansatz restricts the tangent space of $\mathbb{S}^1$-invariant $G_2$-structures by imposing constraints on the tuple $(\dot h, \dot\eta, \dot\omega, \dot\Omega_+)$. The condition $h=1$ implies $\dot h = 0$, ensuring that the induced metric $g_\varphi$ preserves the orthogonal splitting
\begin{equation}
    g_\varphi = \eta^2 + \pi^*g_N,
\end{equation}
where the vertical and horizontal components are determined by the connection $\eta$ and the transverse metric $g_N$, respectively. In this framework, the transverse $SU(3)$-structure is not assumed to be K\"{a}hler; instead, it satisfies the structure equations
\begin{equation}
    d\Omega = \theta \wedge \Omega, \qquad d\omega = \frac{2}{3}\theta \wedge \omega + \nu_3,
\end{equation}
where $\theta$ is a basic $1$-form and $\nu_3$ is the primitive torsion component of type $W_3$. The curvature of the connection is given by $d\eta = F_0^{1,1} + \lambda\omega$, allowing for transverse metric variations that preserve both the torsion class and the curvature decomposition. Differentiating the curvature condition yields the linearized constraints:
\begin{align}
    d\dot\eta &= \dot F_0^{1,1} + \dot\lambda\omega + \lambda\dot\omega, \\
    d\dot\Omega &= \dot\theta \wedge \Omega + \theta \wedge \dot\Omega, \\
    d\dot\omega &= \frac{2}{3}(\dot\theta \wedge \omega + \theta \wedge \dot\omega) + \dot\nu_3.
\end{align}
Furthermore, the closedness condition $dF=0$ (where $F = F_0^{1,1} + \lambda\omega$) must be maintained, requiring:
\begin{equation}
    d\dot F_0^{1,1} + d\dot\lambda \wedge \omega + d\lambda \wedge \dot\omega + \dot\lambda d\omega + \lambda d\dot\omega = 0.
\end{equation}

In summary, this ansatz defines a constrained variation space where the reduced data are captured by the tuple $(g_N, \theta, \lambda, \nu_3, F_0^{1,1})$. An admissible infinitesimal variation is represented by the vector $V = (k, \beta, f, \mu, \rho)$, where $k=\dot g$, $\beta=\dot\theta$, $f=\dot\lambda$, $\mu=\dot\nu_3$, and $\rho=\dot F_0^{1,1}$. To define the gradient flow for the reduced functional $\mathcal{F}(g_N, \theta, \lambda, \nu_3, F_0^{1,1})$, we endow the space of variations with the natural $L^2$-pairing:
\begin{equation}\label{Eq.MetricL2}
    \langle V_1, V_2 \rangle_{L^2} = \int_N \left( \langle k_1, k_2 \rangle_g + \langle \beta_1, \beta_2 \rangle_g + f_1f_2 + \langle \mu_1, \mu_2 \rangle_g + \langle \rho_1, \rho_2 \rangle_g \right) dV_g.
\end{equation}
The formal gradient $\nabla\mathcal{F}$ is then the unique element satisfying:
\begin{equation}
    \frac{d}{dt}\bigg|_{t=0} \mathcal{F} = \left\langle \nabla\mathcal{F}, V \right\rangle_{L^2}
\end{equation}
for every admissible variation $V$.

\begin{theorem}\label{Th.Grad.flow_W345}
Let $\mathcal{F}$ be the reduced $G_2$-Hilbert functional within the constant fiber-length non-K\"{a}hler transverse ansatz \eqref{Eq.structural-ansatz}
With respect to the $L^2$-pairing \eqref{Eq.MetricL2}, the positive gradient flow is given by the coupled system
\begin{equation}\label{Eq.sistem_W345}
\begin{cases}
\displaystyle \partial_t g = -2\pi\left[ \left( \frac{15}{16}\lambda^2 + \frac{5}{36}|\theta|^2 + \frac{1}{8}|\nu_3|^2 + \frac{1}{8}|F_0^{1,1}|^2 \right)g -\frac{5}{18}\theta\otimes\theta -\frac{1}{4}(\nu_3\circ\nu_3) -\frac{1}{4}(F_0^{1,1}\circ F_0^{1,1}) \right], \\[1.5ex]
\displaystyle \partial_t\theta = -\frac{10\pi}{9}\theta, \\[1.5ex]
\displaystyle \partial_t\lambda = -\frac{15\pi}{2}\lambda, \\[1.5ex]
\displaystyle \partial_t\nu_3 = -\pi\nu_3, \\[1.5ex]
\displaystyle \partial_t F_0^{1,1} = -\pi F_0^{1,1}.
\end{cases}
\end{equation}
Here, the composition terms are defined as $(\nu_3\circ\nu_3)_{ij} = \frac{1}{2}(\nu_3)_{iab}(\nu_3)_j{}^{ab}$ and $(F_0^{1,1}\circ F_0^{1,1})_{ij} = (F_0^{1,1})_{ia}(F_0^{1,1})_j{}^{a}$.

For initial data $(g_0,\theta_0,\lambda_0,\nu_{3,0},F_{0,0}^{1,1})$, the Lee form, the trace curvature component, and the primitive torsion and curvature components evolve exponentially:
\begin{equation}\label{eq.res.theta.lambda.nu.F}
\theta(t) = e^{-\frac{10\pi}{9}t}\theta_0, \quad \lambda(t) = e^{-\frac{15\pi}{2}t}\lambda_0, \quad \nu_3(t) = e^{-\pi t}\nu_{3,0}, \quad F_0^{1,1}(t) = e^{-\pi t}F_{0,0}^{1,1}.
\end{equation}
Consequently, the evolution of the metric $g(t)$ is governed by
\begin{equation}\label{Eq.metric-flow-expanded_W345}
\begin{split}
\partial_t g = -2\pi\bigg[ & \left( \frac{15}{16}e^{-15\pi t}\lambda_0^2 + \frac{5}{36}e^{-\frac{20\pi}{9}t}|\theta_0|_{g(t)}^2 + \frac{1}{8} e^{-2\pi t}|\nu_{3,0}|_{g(t)}^2 + \frac{1}{8} e^{-2\pi t}|F_{0,0}^{1,1}|_{g(t)}^2 \right)g(t) \\
&- \frac{5}{18}e^{-\frac{20\pi}{9}t}\theta_0\otimes\theta_0 - \frac{1}{4} e^{-2\pi t}\bigl(\nu_{3,0}\circ_{g(t)}\nu_{3,0}\bigr) - \frac{1}{4} e^{-2\pi t}\bigl(F_{0,0}^{1,1}\circ_{g(t)}F_{0,0}^{1,1}\bigr) \bigg].
\end{split}
\end{equation}
\end{theorem}

\begin{proof}
It fllows of Theorem~\ref{prop.first variation.BW} and using the $L^2$-norm \eqref{Eq.MetricL2}. 
\end{proof}

\begin{cor}
Along the positive gradient flow of the reduced $G_2$--Hilbert functional within the constant fiber-length non-K\"ahler transverse ansatz, the Lee form, the trace curvature component, the primitive torsion component, and the primitive curvature component evolve according to \eqref{eq.res.theta.lambda.nu.F}. Consequently, the tuple
\[
\left(\theta(t),\lambda(t),\nu_3(t),F_0^{1,1}(t)\right)
\]
decays exponentially to zero as $t\to+\infty$. In this limit, the flow drives the ansatz toward the rigid conditions
\begin{equation}
    d\eta=0,
    \qquad
    d\Omega=0,
    \qquad
    d\omega=0.
\end{equation}
Furthermore, the trace of the intrinsic torsion tensor is given by
\[
\operatorname{Tr}(T_{\varphi(t)})=\frac32\lambda(t),
\]
and therefore vanishes asymptotically as the flow approaches its formal limit.
\end{cor}

\begin{remark}
The evolution equation for $g(t)$ is not explicitly integrable in general, as the norms $|\theta(t)|_{g(t)}^2$, $|\nu_3(t)|_{g(t)}^2$, and $|F_0^{1,1}(t)|_{g(t)}^2$, along with the symmetric tensors $\nu_3(t)\circ_{g(t)}\nu_3(t)$ and $F_0^{1,1}(t)\circ_{g(t)} F_0^{1,1}(t)$, depend nonlinearly on the evolving metric through index contractions. Consequently, despite the explicit exponential time dependence of the forms $\theta(t)$, $\nu_3(t)$, $F_0^{1,1}(t)$ and the function $\lambda(t)$, the metric equation remains a fundamentally nonlinear and coupled evolution problem.
\end{remark}

\subsection{Gibbons--Hawking-Type Ansatz}

Over the ansatz \eqref{Eq.G2-ansatz}, we assume that $\varphi$ is closed. Since $\mathcal{L}_\xi\varphi=0$ and $d\varphi=0$, Cartan's formula implies
\[
d(\iota_\xi\varphi)=0.
\]
In particular, the $2$--form determined by the contraction of $\varphi$ with the vector field generating the $\mathbb S^1$--action is closed. Therefore,
\begin{equation}
\begin{split}
d\omega &= 0,\\
d\Omega_+
&=
-\frac{3}{4}h^{-1}dh\wedge\Omega_+
-
h^{-\frac{3}{4}}d\eta\wedge\omega.
\end{split}
\end{equation}
Since $\varphi\wedge\omega$ is closed, it follows that
$d\eta\wedge\omega^2=0.$
Also, we have that
$dV_{\varphi}=\frac{1}{6}h\,\eta\wedge\omega^3.$
Note that if $\varphi$ is a closed $G_2$--structure, then, in the standard torsion decomposition, the torsion forms $\tau_0$, $\tau_1$, and $\tau_3$ vanish, and the only possible non-vanishing torsion component is
$$\tau_2=\tau\in\Omega^2_{14}\simeq\mathfrak g_2.$$

In \cite{Bryant2006}, it is shown that closed $G_2$--structures satisfy the following identities.
\begin{align}
    d\ast_{\varphi}\varphi=&\tau_2\wedge\varphi\\
    d\tau_2=&\frac{1}{7}|\tau_2|^2\varphi+\gamma,\quad\gamma\in\Omega^3_{27}\\
    \Ric(g_{\varphi})=&\frac{1}{4}|\tau_2|^2g_{\varphi}-\frac{1}{4}j(d\tau_2-\frac{1}{2}\ast_{\varphi}(\tau_2\wedge\tau_2))\\
    \scal(g_{\varphi})=&-\frac{1}{2}|\tau_2|^2.
\end{align}

Now, in \cite{gianniotis2025} is defined the $G_2$-Hilbert functional for $G_2$--structures as
\begin{equation}
\begin{matrix}
 \cF: &\Omega^3_+(M) &\rightarrow&\R\\
   & \varphi&\longmapsto&\cF(\varphi)=\int_M\Big(
\frac{1}{6}\scal-\frac{1}{3}|T|^2-\frac{1}{6}(\tr T)^2\Big)dV_{g_{\varphi}}
\end{matrix}
\end{equation}

\begin{prop}
Under the Gibbons-Hawkins assumptions , the Einstein--Hilbert functional
\[
\mathcal \cF(\varphi):=\frac{3}{2}\int_M \scal(g_\varphi)\,\vol_\varphi
\]
satisfies
\[
\mathcal \cF(\varphi)\le 0.
\]
Moreover, the following are equivalent:
\begin{enumerate}
    \item $\mathcal \cF(\varphi)=0$;
    \item
    $
    d_\omega^*(h^{1/2}\Omega_+)=0,
    \qquad
    d^c(h^{-1})=h^{-2}J\gamma_6^1;
    $
    \item the torsion $2$-form $\tau$ of $\varphi$ vanishes;
    \item $\varphi$ is torsion-free.
\end{enumerate}
In particular, torsion-free $\mathbb{S}^1$-invariant closed $G_2$-structures are global maximizers of $\cF(\varphi)$.
\end{prop}
\begin{proof}
    In \cite{fowdar}, they express the torsion form as 
$$\tau_2=\tau_h+\eta\wedge\tau_v$$
for a $2$-form $\tau_h$ and $1$-form $\tau_v$ which are both basic, i.e. they are horizontal and $\mathbb{S}_{\xi}^1$-invariant. So, the norm of the tensor torsion is 
$$|\tau_2|^2_{\varphi}=h^{-2}|\tau_h|_\omega^2+h|\tau_v|_\omega^2.$$
 
As $\tau \in \Lambda_{14}^2$ it follows that

$$
\begin{aligned}
\tau_h \wedge \omega^2 & =0 \\
\tau_v \wedge \frac{1}{2} h^{3 / 2} \omega^2 & =\tau_h \wedge \Omega^{-}
\end{aligned}
$$
and also it was shown that 
\[
\tau_h=d_\omega^*(h^{1/2}\Omega_+),
\qquad
\tau_v=2\,d^c(h^{-1})-2h^{-2}J\gamma_6^1,
\]

Hence
\[
\scal(g_\varphi)
=
-\frac12\left(
h^{-2}\big|d_\omega^*(h^{1/2}\Omega_+)\big|_{g_\omega}^2
+
h\big|2\,d^c(h^{-1})-2h^{-2}J\gamma_6^1\big|_{g_\omega}^2
\right).
\]
Therefore, the proof is complete.
\end{proof}

\begin{theorem}
Let $M^7$ be endowed with a free $\mathbb S^1$--action generated by a vector field $\xi$, and let
\[
\pi:M\to N:=M/\mathbb S^1
\]
be the quotient map. Assume that $\varphi$ is a closed $\mathbb S^1$--invariant $G_2$--structure on $M$ of the form 
\begin{equation}\label{Eq.GHansatz}
  \varphi=\eta\wedge \pi^*\omega+h^{3/4}\pi^*(\Omega_+),  
\end{equation}
where $h>0$ is basic, $\eta$ is an $\mathbb S^1$--invariant connection $1$--form, and $(\omega,\Omega)$ is an $SU(3)$--structure on $N$, with $\Omega_+=\Re\Omega$. Then
\[
g_\varphi=h^{-1}\eta^2+h^{1/2}g_\omega,
\qquad
dV_{g_\varphi}=h\,\eta\wedge dV_{g_\omega},
\]
and
\[
\mathcal{F}(\varphi)
=
2\pi\int_N
\left(
\frac12 h^{1/2}\,\scal(g_\omega)
-\frac32\,\Delta_{g_\omega}(h^{1/2})
-\frac18 h^{-1}|d\eta|_{g_\omega}^2
\right)dV_{g_\omega}.
\]
If moreover $N$ is compact without boundary, then
\begin{equation}\label{eq.GHfunctional}
   \mathcal{F}(\varphi)
=
2\pi\int_N
\left(
\frac12 h^{1/2}\,\scal(g_\omega)
-\frac18 h^{-1}|d\eta|_{g_\omega}^2
\right)dV_{g_\omega}. 
\end{equation}
\end{theorem}
\begin{proof}
From the ansatz \eqref{Eq.G2-ansatz}, the associated metric is $g_\varphi=h^{-1}\eta^2+h^{1/2}g_\omega = h^{1/2}\widetilde g$, where $\widetilde g=g_\omega+f^2\eta^2$ with $f:=h^{-3/4}$. We will use O'Neill formulas of Riemannian submersions \cite[Chapter 11]{Besse2008} to concluding that the scalar curvature is $\scal(\widetilde g) = \scal(g_\omega) -\frac14 f^2|d\eta|_{g_\omega}^2 -2f^{-1}\Delta_{g_\omega}f$. Using $f=e^{-\frac34\log h}$ and the identity $f^{-1}\Delta f = \Delta(\log f) + |d\log f|^2$, we find
\[
\scal(\widetilde g) = \scal(g_\omega) -\frac14 h^{-3/2}|d\eta|_{g_\omega}^2 +\frac32\Delta_{g_\omega}(\log h) -\frac98|d\log h|_{g_\omega}^2.
\]
Applying the conformal change $g_\varphi=e^{2u}\widetilde g$ with $u=\frac14\log h$, we have $$\scal(g_\varphi) = h^{-1/2} ( \scal(\widetilde g) -3\Delta_{\widetilde g}(\log h) -\frac{15}{8}|d\log h|_{\widetilde g}^2 )$$
Since $h$ is basic, we use $\Delta_{\widetilde g}(\log h) = \Delta_{g_\omega}(\log h) -\frac34|d\log h|_{g_\omega}^2$ and $|d\log h|_{\widetilde g}^2 = |d\log h|_{g_\omega}^2$ to simplify the expression to
\[
\scal(g_\varphi) = h^{-1/2}\scal(g_\omega) -\frac32 h^{-1/2}\Delta_{g_\omega}(\log h) -\frac34 h^{-1/2}|d\log h|_{g_\omega}^2 -\frac14 h^{-2}|d\eta|_{g_\omega}^2.
\]
Using $\Delta_{g_\omega}(h^{1/2}) = h^{1/2} ( \frac12\Delta_{g_\omega}(\log h) + \frac14|d\log h|_{g_\omega}^2 )$ and $dV_{g_\varphi}=h\,\eta\wedge dV_{g_\omega}$, it becomes $\scal(g_\varphi)dV_{g_\varphi} = ( h^{1/2}\scal(g_\omega) -3\Delta_{g_\omega}(h^{1/2}) -\frac14 h^{-1}|d\eta|_{g_\omega}^2 ) \eta\wedge dV_{g_\omega}$. Integrating over the $\mathbb S^1$--fibers (volume $2\pi$), the reduced functional is
\[
\mathcal{F}(\varphi) = 2\pi\int_N \left( \frac12 h^{1/2}\scal(g_\omega) -\frac32\Delta_{g_\omega}(h^{1/2}) -\frac18 h^{-1}|d\eta|_{g_\omega}^2 \right)dV_{g_\omega}.
\]
Finally, if $N$ is compact without boundary, the term $\int_N\Delta_{g_\omega}(h^{1/2})\,dV_{g_\omega}$ vanishes, and we conclude the proof.
\end{proof}

We consider the negative gradient flow of $\mathcal F$ because the gradient is the direction of steepest increase of the functional. 

\begin{theorem}\label{Th.gradientGH}
Let $\mathcal F$ be the $G_2$--Hilbert functional \eqref{eq.GHfunctional} associated with the Gibbons--Hawking-type ansatz \eqref{Eq.GHansatz}, where $g=g_\omega$ and $F=d\eta$. Then the formal negative $L^2$--gradient flow of $\mathcal F$ is
\[
\left\{
\begin{aligned}
\partial_t g
&=
\frac12 h^{1/2}\Big(\Ric-\frac12\scal\, g\Big)
-\frac12\Big(\text{Hess}(h^{1/2})-(\Delta h^{1/2})\,g\Big)
-\frac18 h^{-1}\Big(j(F)-\frac12 |F|^2 g\Big),
\\[1ex]
\partial_t h
&=
-\frac14 h^{-1/2}\scal(g)
-\frac18 h^{-2}|F|^2,
\\[1ex]
\partial_t \eta
&=
\frac14\, d^*(h^{-1}F),
\end{aligned}
\right.
\]
where $j(F)_{ij}=F_{ik}F_j{}^k.
$
\end{theorem}

\begin{proof}
By \eqref{eq.GHfunctional}, on a compact quotient $N$ the reduced functional is
\[
\mathcal F(g,h,\eta)
=
2\pi\int_N
\left(
\frac12 h^{1/2}\scal(g)
-\frac18 h^{-1}|F|^2
\right)dV_g,
\qquad
F=d\eta.
\]
Since the factor $2\pi$ only rescales the time parameter of the gradient flow, we compute the formal gradient for the functional without this global factor. Let
$\dot g=k,
$, $
\dot h=v,
$ and $
\dot\eta=\alpha.
$
Then $\dot F=d\alpha.
$ Let $u=h^{1/2}$ and  evolving the term using the tangent space of the metric $g$, we have
\[
\int_N \frac12 u\,\scal(g)\,dV_g.
\]
Using the standard variation formulas
\[
\delta dV_g=\frac12\tr_g(k)dV_g,
\]
and
\[
\delta\scal(g)
=
-\langle \Ric,k\rangle
+
\operatorname{div}\operatorname{div}k
-
\Delta(\tr_g k),
\]
and integrating by parts, we obtain
\begin{equation}
\delta\left(
\int_N \frac12 u\,\scal(g)\,dV_g
\right)
=
\int_N
\Big\langle
-\frac12 u\,\Ric
+\frac12\text{Hess} u
-\frac12(\Delta u)g
+\frac14 u\,\scal(g)g,
k
\Big\rangle dV_g 
+
\int_N
\frac14 h^{-1/2}\scal(g)\,v\,dV_g.
\end{equation}
Now we evolve in the direction of the function $h$.
$\dot u=\frac12 h^{-1/2}v$ and we vary the curvature term
\[
-\frac18\int_N h^{-1}|F|^2\,dV_g.
\]
The metric variation of the norm of a $2$--form gives
\[
\delta_g |F|^2=-\langle j(F),k\rangle,
\qquad
j(F)_{ij}=F_{ik}F_j{}^k.
\]
Therefore the metric variation of this term is
\[
\int_N
\left\langle
\frac18 h^{-1}j(F)
-\frac1{16}h^{-1}|F|^2g,
k
\right\rangle dV_g.
\]
The variation with respect to $h$ is
\[
\frac18\int_N h^{-2}|F|^2v\,dV_g.
\]
Finally, the variation with respect to $\eta$ is
\[
-\frac14\int_N h^{-1}\langle F,d\alpha\rangle dV_g.
\]
Integrating by parts gives
\[
-\frac14\int_N h^{-1}\langle F,d\alpha\rangle dV_g
=
-\frac14\int_N
\langle d^*(h^{-1}F),\alpha\rangle dV_g.
\]
Combining these computations, the first variation is
\[
\mathcal G_g
=
-\tfrac12 h^{1/2}\Ric
+\tfrac12\text{Hess}(h^{1/2})
-\tfrac12(\Delta h^{1/2})g
+\tfrac14 h^{1/2}\scal(g)g
+\tfrac18 h^{-1}j(F)
-\tfrac1{16}h^{-1}|F|^2g.
\]
Therefore, we have
\begin{align*}
\delta\mathcal F
={}&
\int_N \langle \mathcal G_g,k\rangle\,dV_g +
\int_N
\left(
\tfrac14 h^{-1/2}\scal(g)
+\tfrac18 h^{-2}|F|^2
\right)v\,dV_g -
\tfrac14\int_N
\langle d^*(h^{-1}F),\alpha\rangle\,dV_g.
\end{align*}
Thus, with respect to the natural $L^2$--pairing in the variables
$(g,h,\eta),
$ the formal gradient is
\[
\nabla_g\mathcal F
=
-\frac12 h^{1/2}\Ric
+\frac12\text{Hess}(h^{1/2})
-\frac12(\Delta h^{1/2})g
+\frac14 h^{1/2}\scal(g)g
+\frac18h^{-1}j(F)
-\frac1{16}h^{-1}|F|^2g,
\]
\[
\nabla_h\mathcal F
=
\frac14 h^{-1/2}\scal(g)
+\frac18h^{-2}|F|^2,
\]
and
\[
\nabla_\eta\mathcal F
=
-\frac14 d^*(h^{-1}F).
\]
This proves the theorem.
\end{proof}

\begin{remark}
 The tensor $\text{Ric} - \frac{1}{2}\text{scal}\,g$ is the classical Einstein tensor. As the variational gradient of the Einstein-Hilbert action, it is a central object in Riemannian geometry. Therefore, in this restricted flat-fibration case, the metric evolution completely decouples from the connection and is governed entirely by a time-dependent multiple of the Einstein tensor, shifting the analytical focus to classical Einstein-type equations.
\end{remark}
\begin{cor}\label{cor:flow_consequences}
Let $(g(t),h(t),\eta(t))$ be a solution to the $L^2$--gradient flow defined in Theorem \ref{Th.gradientGH}. Assume the base manifold $N$ is closed and has dimension $6$. Then the flow exhibits the following properties:
\begin{enumerate}
    \item The de Rham cohomology class $[F]\in H^2_{\mathrm{dR}}(N)$ of the curvature $2$--form $F=d\eta$ is preserved along the flow.

    \item If the flow admits a stationary limit, the connection $\eta$ satisfies
    $$
    d^*(h^{-1}F)=0,
    $$
    and the scalar curvature of the base manifold satisfies
    $$
    \scal(g)=-\frac12 h^{-3/2}|F|^2\leq 0.
    $$

    \item The evolution of the volume form $d\vol_g$ of the base manifold $N$ is given by
    $$
    \partial_t(d\vol_g)
    =
    \left(
    -\frac12 h^{1/2}\scal(g)
    +
    \frac54\Delta(h^{1/2})
    +
    \frac18 h^{-1}|F|^2
    \right)d\vol_g.
    $$
\end{enumerate}
\end{cor}

\begin{proof}
\begin{enumerate}
\item Taking the exterior derivative of the third equation in the flow gives
    $$
    \partial_t F
    =
    d(\partial_t\eta)
    =
    \frac14 d\left(d^*(h^{-1}F)\right).
    $$
    Since the right-hand side is exact, the cohomology class $[F]$ remains constant along the flow.

    \item At a stationary point, all time derivatives vanish. From $\partial_t\eta=0$, we obtain
    $$
    d^*(h^{-1}F)=0.
    $$
    From $\partial_t h=0$, using the second equation of Theorem \ref{Th.gradientGH}, we get
    $$
    -\frac14 h^{-1/2}\scal(g)-\frac18h^{-2}|F|^2=0.
    $$
    Hence
    $$
    \scal(g)
    =
    -\frac12 h^{-3/2}|F|^2.
    $$
    Since $h>0$ and $|F|^2\geq0$, it follows that
    $$
    \scal(g)\leq0.
    $$

    \item For a time-dependent metric $g(t)$, the variation of the Riemannian volume form is
    $$
    \partial_t(d\vol_g)
    =
    \frac12\tr_g(\partial_t g)\,d\vol_g.
    $$
    We compute the trace of the metric evolution equation. In dimension $6$,
    $$
    \tr_g(g)=6,
    \qquad
    \tr_g(\Ric)=\scal(g),
    \qquad
    \tr_g(\text{Hess}(h^{1/2}))=\Delta(h^{1/2}),
    $$
    and
    $$
    \tr_g(j(F))=|F|^2.
    $$
    Therefore,
    \begin{align*}
    \tr_g(\partial_t g)
 =
    -h^{1/2}\scal(g)
    +
    \frac52\Delta(h^{1/2})
    +
    \frac14 h^{-1}|F|^2.
    \end{align*}
\end{enumerate}
\end{proof}
\begin{prop}[Rigidity of the unnormalized flow]\label{prop.rig.gh}
Let $N$ be a closed $6$--dimensional manifold. The only stationary solutions to the unnormalized $L^2$--gradient flow of $\widetilde{\mathcal F}$ are trivial configurations where:
\begin{enumerate}
    \item The principal $\mathbb S^1$--bundle is flat, that is, $F\equiv 0$.
    \item The base manifold is scalar-flat, that is, $\scal(g)=0$.
    \item The fiber-length function is constant, that is, $h=\mathrm{constant}$.
\end{enumerate}
\end{prop}

\begin{proof}
Assume that the flow has reached a stationary limit, meaning
$$
\partial_t g=0,
\qquad
\partial_t h=0,
\qquad
\partial_t\eta=0.
$$

Since the metric is stationary, the evolution of the Riemannian volume form must vanish:
$$
\partial_t(d\vol_g)=0.
$$
By Corollary \ref{cor:flow_consequences}(3), this gives
$$
0
=
-\frac12 h^{1/2}\scal(g)
+
\frac54\Delta(h^{1/2})
+
\frac18 h^{-1}|F|^2.
$$
Moreover, the stationary condition for the function $h$, namely $\partial_t h=0$, gives
$$
\scal(g)
=
-\frac12 h^{-3/2}|F|^2.
$$
Substituting this into the volume trace equation, we obtain
\begin{align*}
0
&=
-\frac12 h^{1/2}
\left(
-\frac12 h^{-3/2}|F|^2
\right)
+
\frac54\Delta(h^{1/2})
+
\frac18h^{-1}|F|^2
\\
&=
\frac14 h^{-1}|F|^2
+
\frac54\Delta(h^{1/2})
+
\frac18h^{-1}|F|^2
\\
&=
\frac38 h^{-1}|F|^2
+
\frac54\Delta(h^{1/2}).
\end{align*}
Integrating over the closed manifold $N$ gives
$$
\int_N
\left(
\frac38 h^{-1}|F|^2
+
\frac54\Delta(h^{1/2})
\right)d\vol_g
=
0.
$$
Since $N$ is closed,
$$
\int_N \Delta(h^{1/2})\,d\vol_g=0.
$$
Hence
$$
\frac38
\int_N h^{-1}|F|^2\,d\vol_g
=
0.
$$
Because $h>0$ and $|F|^2\geq0$, it follows that
$F\equiv0.
$ Substituting $F=0$ into the stationary constraint for $h$, we get
$$
\scal(g)=0.
$$
The volume trace equation then reduces to
$$
\Delta(h^{1/2})=0.
$$
Since $N$ is closed, $h^{1/2}$ is constant. Therefore $h$ is constant. Thus, any stationary solution of the unnormalized flow has flat connection, scalar-flat base metric, and constant fiber length.
\end{proof}

\bibliographystyle{plain}
\bibliography{Bibliografia-2020-07}

\end{document}